\documentclass[a4paper]{article}
\usepackage[top=30truemm, bottom=30truemm, left=30truemm, right=30truemm]{geometry}
\usepackage{amsmath, amssymb}
\usepackage{amsthm}
\usepackage{mathtools}
\usepackage{proof}
\theoremstyle{definition}
\newtheorem{theorem}{Theorem}

\newtheorem{lemma}[theorem]{Lemma}

\newtheorem{proposition}[theorem]{Proposition}
\newtheorem{remark}{Remark}
\usepackage{cases}
\usepackage{url}
\usepackage{authblk}
\title{Ground states for $p$-fractional Choquard-type equations with critical local nonlinearity and doubly critical nonlocality}
\author{Masaki Sakuma\thanks{Email: masakisakuma0110@gmail.com}}
\affil{Graduate School of Mathematical Sciences, The University of Tokyo, \\ Meguro-ku, Tokyo, Japan}

\begin{document}
\maketitle
\begin{abstract}
We consider a $p$-fractional Choquard-type equation 
\[
(-\Delta)_p^s u+a|u|^{p-2}u=b(K\ast F(u))F'(u)+\varepsilon_g |u|^{p_g-2}u \quad\text{in $\mathbb{R}^N$},
\]
where $0<s<1<p<p_g\leq p_s^*$, $N\geq \max\{2ps+\alpha, p^2 s\}$, $a,b,\varepsilon_g\in (0,\infty)$, $K(x)= |x|^{-(N-\alpha)}$, $\alpha\in (0,N)$ and $F(u)$ is a doubly critical nonlinearity in the sense of the Hardy-Littlewood-Sobolev inequality. It is noteworthy that the local nonlinearity may also have critical growth. Combining Brezis-Nirenberg's method with some new ideas, we obtain ground state solutions via the mountain pass lemma and a new generalized Lions-type theorem.

\vspace{1ex}\par
{\flushleft{{\bf Keywords:} Choquard equation; Fractional $p$-Laplacian; Critical exponents; Variational method; Mountain pass lemma; Lions-type theorem}}
\end{abstract}
\section{Introduction}
In the present paper, we study a $p$-fractional Choquard-type equation
\begin{equation}\label{cho}
(-\Delta)_p^s u+a|u|^{p-2}u=b(K\ast F(u))F'(u)+\varepsilon_g |u|^{p_g-2}u \quad\text{in $\mathbb{R}^N$},
\end{equation}
where $0<s<1<p<p_g\leq p_s^*$; $N\geq \max\{2ps+\alpha, p^2 s\}$; $a,b,\varepsilon_g>0$ are positive constants; $(-\Delta)_p^s$ denotes the fractional $p$-Laplace operator, which is defined up to a normalization factor as
\[
(-\Delta)^s_p u(x)\coloneqq\lim_{\varepsilon\to +0}\int_{\mathbb{R}^N\setminus B_\varepsilon (x)}\frac{|u(x)-u(y)|^{p-2}(u(x)-u(y))}{|x-y|^{N+ps}}dy;
\]
and $K(x)=K_\alpha (x)=|x|^{-(N-\alpha)}$ ($\alpha\in (0,N)$) is the Riesz potential up to a scaling factor. In addition, 
\[
F(u)=\displaystyle\frac{1}{p^\downarrow}|u|^{p^\downarrow}+ \frac{1}{p^\uparrow}|u|^{p^\uparrow},
\]
where $p^\downarrow=\displaystyle\frac{N+\alpha}{2N}p$ and $p^\uparrow=\displaystyle\frac{p_s^*}{p}p^\downarrow=\frac{N+\alpha}{N-sp}\frac{p}{2}$ are respectively lower and upper critical exponents in the sense of the Hardy-Littlewood-Sobolev inequality, and $p_s^*=\displaystyle\frac{pN}{N-ps}$ denotes the fractional Sobolev critical exponent. Although the notation with $\uparrow$ and $\downarrow$ is not common, we use this here for the convenience of distinguishing the three criticalities (upper criticality, lower criticality and Sobolev criticality) by three symbols $\uparrow$, $\downarrow$ and $\ast$.\par
Throughout this paper, for $q\in [1,\infty]$, $\|\cdot\|_q$ denotes the $L^q$ norm, $L^q_{+}= L^q_{+} (\mathbb{R}^N)$ denotes the set consisting of all positive $L^q$ functions, $B_r(x)$ denotes the open ball with radius $r$ centered at $x$ in $\mathbb{R}^N$, and $C$, $C'$, $C''$, $C_i$ and $C_i'$ are various positive constants.\par

The associated functional is
\[
I[u]=\frac{1}{p}[u]_{s,p}^p+\frac{1}{p}a\int_{\mathbb{R}^N}|u|^p dx-\frac{b}{2}\int_{\mathbb{R}^N}(K\ast F(u))F(u)dx-\frac{\varepsilon_g}{p_g} \int_{\mathbb{R}^N} |u|^{p_g} dx,
\]
where $[u]_{s,p}$ denotes the Gagliardo semi-norm of the homogeneous fractional Sobolev-Slobodeckij space $D^{s,p}(\mathbb{R}^N)$, which is defined by
\[
[u]_{s,p}^p\coloneqq \int_{\mathbb{R}^N} \int_{\mathbb{R}^N}\frac{|u(x)-u(y)|^p}{|x-y|^{N+ps}}dxdy,
\]
and we work in the inhomogeneous fractional Sobolev-Slobodeckij space $W^{s,p} (\mathbb{R}^N)$ equipped with the norm $\|\cdot\|=\|\cdot\|_{W^{s,p}}\coloneqq (\|\cdot\|_p^p+ [\cdot ]_{s,p}^p)^{1/p}$. A critical point of $I$ is called a weak solution of \eqref{cho}. \par
Our main theorem is as follows:
\begin{theorem}\label{main}
Assume $N\geq \max\{2ps+\alpha, p^2 s\}$, $0<s<1$ and $a,b>0$. Then, the equation \eqref{cho} has a nontrivial weak solution if one of the following conditions holds:
\begin{itemize}
\item $1<p<\infty$, $p<p_g< p_s^*$ and $\varepsilon_g>0$.
\item $p=2$, $p_g= p_s^*$ and $\varepsilon_g>0$ is sufficiently small.
\end{itemize}
\end{theorem}
In recent years, great attention has been drawn to the study of Choquard-type equations and fractional versions of elliptic equations. The nonlinear Choquard equation 
\begin{equation}
-\Delta u+V(x)u=(I_\alpha \ast |u|^q)|u|^{q-2}u\quad\text{in }\mathbb{R}^N
\end{equation}
first appeared under $N=3$ and $q=2$ in the description of the quantum theory of polaron by Pekar \cite{Pekar} in 1954 and the modeling of an electron trapped in its own hole in Choquard's work in 1976, which was obtained as an approximation related to the Hartree-Fock theory concerning one-component plasma in \cite{Hartree-Fock}. When $V\equiv 1$, the ground state solutions exist if $2^\downarrow<q<2^\uparrow$ due to the mountain pass lemma or the method of the Nehari manifold, and their qualitative properties and decay asymptotics are studied in \cite{Moroz 2} while there are no nontrivial solutions if $q=2^\downarrow$ or if $q=2^\uparrow$ as a consequence of the Pohozaev identity. For more general $V$, Lions' concentration compactness method is helpful. Regarding the lower critical case $q=2^\downarrow$, Moroz and Schaftingen \cite{Moroz lower critical} showed a sufficient condition concerning $V$ for the existence of ground states. Regarding the upper critical case, Li and Tang \cite{Existence of a ground state solution for Choquard equation with the upper critical exponent} considered the Choquard-type equation
\[
-\Delta u+u=(I_\alpha \ast |u|^q)|u|^{q-2}u+g(u)\quad\text{in }\mathbb{R}^N
\]
involving a local nonlinearity $g(u)$ satisfying some subcritical growth conditions and obtained ground state solutions. For more information on the various results related to the non-fractional Choquard-type equations and their variants, see \cite{Moroz}. On the other hand, in the field of fractional quantum mechanics, the nonlinear fractional Schr\"{o}dinger equation was first proposed by Laskin \cite{Laskin} as a consequence of expanding the classical Feynman path integral to the L\'{e}vy-like quantum mechanical paths. The stationary states of the corresponding fractional Schr\"{o}dinger-Newton equations satisfy the fractional Choquard equations. d'Avenia et al. \cite{classical fractional Choquard 1, AttainTalenti} studied the existence and some properties of the weak solutions for the fractional subcritical Choquard equation 
\[
(-\Delta)^{s} u+V(x)u=(I_\alpha \ast |u|^q)|u|^{q-2}u\quad\text{in }\mathbb{R}^N.
\]
As for the Hardy-Littlewood-Sobolev doubly critical case, Seok \cite{doubly critical Seok} obtained a nontrivial solution to
\begin{equation}\label{non-fractional doubly critical Choquard}
-\Delta u+u=(K\ast F(u))F'(u)\quad \text{in }\mathbb{R}^N.
\end{equation}
This result was later extended by \cite{doubly critical new}, and Su et al. \cite{doubly critical fractional} considered the fractional version 
\[
(-\Delta)^s u+u=(K\ast F(u))F'(u)\quad \text{in }\mathbb{R}^N
\]
and its variant
\[
\varepsilon^{2s} (-\Delta)^s u+V(x)u=\varepsilon^{-\alpha}(K\ast F(u))F'(u)\quad \text{in }\mathbb{R}^N
\]
with a parameter $\varepsilon>0$ and analyzed the concentration behavior of its solutions. Liu et al. \cite{doubly critical infinitely many} obtained infinitely many solutions for \eqref{non-fractional doubly critical Choquard} by using the notion of the Krasnoselskii genus. These all deal with the double criticality in the sense that $F$ involves both upper and lower Hardy-Littlewood-Sobolev critical exponents $2^\uparrow$ and $2^\downarrow$. On the other hand, there is another meaning of ``doubly critical'' for Choquard equations. It is the combination of Sobolev criticality of the local nonlinearity $|u|^{p_s^*-2}u$ and either upper or lower Hardy-Littlewood-Sobolev criticality of the nonlocal nonlinearity $(K\ast F(u))F'(u)$. Regarding the doubly critical case in this meaning, Cai and Zhang \cite{local and nonlocal doubly critical Brezis-Nirenberg} considered the Brezis-Nirenberg type problem
\[
-\Delta u-\lambda u=\alpha |u|^{2^*-2}u+\beta (K\ast |u|^{2^\uparrow}) |u|^{2^\uparrow -2}u \quad \text{in }\Omega,
\]
where $\Omega$ is a bounded domain in $\mathbb{R}^N$, and Li et al. \cite{Ground State Solutions for a Class of Choquard Equations Involving Doubly Critical Exponents} obtained a ground state solution to the autonomous Choquard equation
\[
-\Delta u+ u=(K\ast |u|^{2^\downarrow}) |u|^{2^\downarrow -2}u+|u|^{2^*-2}u+g(u) \quad \text{in }\mathbb{R}^N
\]
using a minimax principle and the Pohozaev manifold method. However, there is no existence result for ground state solutions to the Choquard equations involving Sobolev criticality and both upper and lower Hardy-Littlewood-Sobolev criticality in the existing literature. In the present paper, motivated by the above works, we deal with such a ``triply critical'' case where two types of double criticality are fused. \par
We prove the existence of nontrivial solutions by the mountain pass lemma, or equivalently, using the method of Nehari manifold. Such solutions are hence ground state solutions. In order to eliminate the possibility of losing the compactness due to the translation invariance of the equation and the scaling property corresponding to the criticality, we focus on the energy levels corresponding to scaling limits and use a new Lions-type theorem. 

\section{Preliminaries}
First, we introduce three important best constants:
\[
S^\downarrow\coloneqq\inf_{u\in W^{s,p}\setminus\{0\}}\frac{\|u\|_{p}^p}{\left(\int_{\mathbb{R}^N}(K\ast |u|^{p^\downarrow}) |u|^{p^\downarrow}dx\right)^{p/(2\cdot p^\downarrow)}},
\]
\[
S^\uparrow=S_{H,L}\coloneqq\inf_{u\in D^{s,p}\setminus\{0\}}\frac{[u]_{s,p}^p}{\left(\int_{\mathbb{R}^N}(K\ast |u|^{p^\uparrow}) |u|^{p^\uparrow}dx\right)^{p/(2\cdot p^\uparrow)}},
\]
\[
S^*\coloneqq \inf_{u\in D^{s,p}\setminus\{0\}}\frac{[u]_{s,p}^p}{\|u\|_{p_s^*}^{p}}.
\]

For general $p\neq 2$, the explicit formula for the extremal functions for the $p$-fractional Sobolev inequality is not known yet, though it is conjectured that it is of the form
\[
U(x)=C(1+|x|^{\frac{p}{p-1}})^{s-\frac{N}{p}}
\]
up to translation and dilation. However, there is a result about the asymptotic behavior of $U$, as seen in \cite{Optimal decay of extremal functions for the fractional Sobolev inequality} and \cite{The Brezis-Nirenberg problem for the fractional p-Laplacian}. 
\begin{proposition}\label{asymptotic behavior}
There exists a radially symmetric decreasing positive minimizer $U$ for $S^*$. For such $U$, there exist constants $c_1,c_2,U_\infty>0$ such that for any $x\in\mathbb{R}^N$, we have
\[
\frac{c_1}{1+|x|^{\frac{N-ps}{p-1}}}\leq U(x)\leq \frac{c_2}{1+|x|^{\frac{N-ps}{p-1}}}
\]
and
\[
|x|^{\frac{N-ps}{p-1}} U(x)\to U_\infty \quad (|x|\to\infty).
\]
\end{proposition}
On the other hand, the extremal functions for the Hardy-Littlewood-Sobolev inequality are already well-known. For example, see \cite{Lieb}. 
\begin{proposition}
For $p_1,p_2>1$ with $1/p_1+(N-\alpha)/N+1/p_2=2$, there exists a sharp constant $C(N,\alpha,p_1)>0$ such that for any $f_1\in L^{p_1}(\mathbb{R}^N)$, $f_2\in L^{p_2}(\mathbb{R}^N)$, we have
\[
\int_{\mathbb{R}^N}(K\ast f_1)f_2 dx\leq C(N,\alpha,p_1)\|f_1\|_{p_1}\|f_2\|_{p_2}.
\]
If $p_1=p_2(=\frac{2N}{N+\alpha})$, then the equality holds if and only if there exist constants $C_1,C_2,C_3$ with $C_3\neq 0$ and $x_0\in\mathbb{R}^N$ such that $f_1=C_1 f_2$ and
\[
f_2(x)= \frac{C_2}{(C_3+|x-x_0|^{2})^{\frac{N+\alpha}{2}}}.
\]
\end{proposition}
Let us define $C(N,\alpha, \frac{2N}{N+\alpha})=C(N,\alpha)$. 
The case $p=2$ is special in the sense that the best constants $S^*$ and $C(N,\alpha)$ are attained at the same time by the same functions, and thus $S^*= C(N,\alpha)^{\frac{p}{2\cdot p^\uparrow}} S^\uparrow$ holds. In fact, if $p=2$, they are, up to a scaling and a translation, of the form
\[
(1+|x|^2)^{-(N-ps)}=((1+|x|^2)^{-\frac{N+\alpha}{2}})^{\frac{1}{p^\downarrow}}.
\]
On the other hand, if $p\neq 2$, the asymptotic behaviors of $U(x)$ and $((1+|x|^2)^{-\frac{N+\alpha}{2}})^{\frac{1}{p^\downarrow}}$ are different from each other. \par
In order to recover the compactness of $(PS)_c$ sequences, we generalize a Lions-type theorem. The classical Lions-type theorem is as follows:
\begin{proposition}
Let $\{u_n\}\subset H^1(\mathbb{R}^N)$ be a bounded sequence. Assume that there exists $r\in (2,2^*)$ such that
\[
\lim_{n\to\infty} \|u_n\|_{r}>0.
\] 
Then, up to a subsequence, there exists $\{y_n\}\subset \mathbb{R}^N$ such that $\{u_n(\cdot +y_n)\}$ converges to some $u\neq 0$ weakly in $H^1$, almost everywhere and in $L^r_\mathrm{loc}(\mathbb{R}^N)$. 
\end{proposition}
On the other hand, our new Lions-type theorem is as follows:
\begin{proposition}\label{CC}
Let $\{u_n\}\subset W^{s,p}(\mathbb{R}^N)$ be a bounded sequence. Assume
\[
\lim_{n\to\infty} \|u_n\|_{p}>0,\quad \lim_{n\to\infty} \|u_n\|_{p_s^*}>0.
\]
Then, up to a subsequence, there exists $\{y_n\}\subset \mathbb{R}^N$ such that $\{u_n(\cdot +y_n)\}$ converges weakly to some $u\neq 0$. 
\end{proposition}
Compared to the classical Lions-type theorem, the assumption of Proposition \ref{CC} is weak because $\|u_n\|_p\to 0$ or $\|u_n\|_{p_s^*}\to 0$ implies $\|u_n\|_r\to 0$ for $r\in (p,p_s^*)$ due to the interpolation inequality. \par
In order to show Proposition \ref{CC}, we introduce the homogenous Sobolev-Slobodeckij space $\dot{W}^{s,p}$, the Riesz potential space (the homogenous version of the Bessel potential space) $\dot{H}^{s,p}$, the Morrey space $\mathcal{M}_q^r$, the homogeneous Besov space $\dot{B}_{r,q}^{s}$ and the homogeneous Triebel-Lizorkin space $\dot{F}_{r,q}^{s}$. Each space is equipped with the following (semi-)norm and defined as the space consisting of every tempered distribution (possibly modulo polynomials) such that the (semi-)norm is finite. 
\begin{gather*}
\mathcal{D}^{s,p}=\dot{W}^{s,p}= \dot{W}^{s,p}(\mathbb{R}^N)\coloneqq\{f\in L^{p_s^*} (\mathbb{R}^N)\mid [f]_{s,p}<\infty\} \quad (ps<N);\\
\|f\|_{\dot{H}^{s,p}}\coloneqq\|\mathcal{F}^{-1}(|\xi|^s\mathcal{F}f)\|_p \quad (s\in\mathbb{R},p\in (1,\infty)),\\ 
\dot{H}^{s,p}= \dot{H}^{s,p}(\mathbb{R}^N)\coloneqq\{f\in\mathcal{S}_\infty '(\mathbb{R}^N)\mid \|f\|_{\dot{H}^{s,p}}<\infty\} \quad (s\in\mathbb{R},p\in (1,\infty));\\
\|f\|_{\mathcal{M}_q^r}\coloneqq\sup_{R>0,x\in\mathbb{R}^N}R^{\frac{N}{r}-\frac{N}{q}}\|f\|_{L^q(B_R(x))} \quad(1\leq q\leq r<\infty),\\
\mathcal{M}_q^r =\mathcal{M}_q^r(\mathbb{R}^N)\coloneqq\{f\in L^1_\mathrm{loc} (\mathbb{R}^N)\mid \|f\|_{\mathcal{M}_q^r}<\infty\}\quad(1\leq q\leq r<\infty);\\
\|f\|_{\dot{B}^{s}_{r,q}}\coloneqq\|\{\|2^{js}\varphi_j(D)f\|_{L^p(\mathbb{R}^N)}\}_{j\in\mathbb{Z}}\|_{\ell^q(\mathbb{Z})} \quad(s\in\mathbb{R}, r,q\in(0,\infty]),\\
\dot{B}^{s}_{r,q}= \dot{B}^{s}_{r,q}(\mathbb{R}^N)\coloneqq\{f\in\mathcal{S}_\infty'(\mathbb{R}^N)\mid \|f\|_{\dot{B}^{s}_{r,q}}<\infty\}\quad(s\in\mathbb{R}, r,q\in(0,\infty]);\\
\|f\|_{\dot{F}^{s}_{r,q}}\coloneqq\|\|\{2^{js}\varphi_j(D)f\}_{j\in\mathbb{Z}}\|_{\ell^q(\mathbb{Z})} \|_{L^p(\mathbb{R}^N)} \quad(s\in\mathbb{R}, r\in (0,\infty),q\in(0,\infty]),\\
\dot{F}^{s}_{r,q}= \dot{F}^{s}_{r,q}(\mathbb{R}^N)\coloneqq\{f\in\mathcal{S}_\infty'(\mathbb{R}^N)\mid \|f\|_{\dot{F}^{s}_{r,q}}<\infty\}\quad(s\in\mathbb{R}, r\in (0,\infty),q\in(0,\infty]),
\end{gather*}
where $\mathcal{F}:\mathcal{S}'(\mathbb{R}^N)\to \mathcal{S}'(\mathbb{R}^N)$ denotes the Fourier transform, $\mathcal{S}'(\mathbb{R}^N)$ [resp. $\mathcal{S}'_\infty(\mathbb{R}^N)$] denotes the space of all tempered distributions [resp. all tempered distributions modulo polynomials] and $\{\varphi_j(D)f\}_{j\in\mathbb{Z}}$ is the Littlewood-Paley decomposition of $f$, that is, $\varphi_j(D)f=\mathcal{F}^{-1}(\varphi_j\mathcal{F}f)$, $\varphi_j(\xi)=\psi(2^{-j}\xi)-\psi(2^{-j+1}\xi)$ ($j\in\mathbb{Z}$) for some real-valued radial smooth function $\psi$ such that $\operatorname{supp}\psi\subset B_2(0)$ and $\psi=1$ on $B_1(0)$. The space $\mathcal{S}_\infty '(\mathbb{R}^N)$ can be identified with the space of every tempered distribution $f$ such that $\sum_{j=-n}^n \varphi_j(D)f$ converges to $f$ in $\mathcal{S}'(\mathbb{R}^N)$ as $n\to\infty$. By the H\"{o}lder's inequality for $\ell^q$, the embedding $\dot{F}^{s}_{r,q_1} \hookrightarrow \dot{F}^{s}_{r,q_2}$ holds for $1\leq q_1\leq q_2$. Moreover, $\dot{F}^s_{p,p}= \dot{B}^s_{p,p}= \dot{W}^{s,p}$ and $\dot{F}^s_{p,2}= \dot{H}^{s,p}$ holds under the standard identification. For details, see \cite{Triebel}. In addition, we use a refined Sobolev inequality with the Besov norm in \cite{Besov} and Lemma 3.4 in \cite{Morrey}. The well-known G\'{e}rard-Meyer-Oru's inequality implies
\begin{equation}\label{GMO}
\begin{split}
\|f\|_{p_s^*}&\leq C\|(-\Delta)^{s/2} f\|_{p}^{\theta}\|f\|_{\dot{B}_{\infty,\infty}^{-N/p_s^*}}^{1-\theta} \\
&= C'\|f\|_{\dot{H}^{s,p}}^{\theta}\|f\|_{\dot{B}_{\infty,\infty}^{-N/p_s^*}}^{1-\theta}
\end{split}
\end{equation}
for $\theta\in [\frac{p}{p_s^*},1)$. However, since we now work in $W^{s,p}$, this inequality is not applicable unless $p\leq 2$. If $p\leq 2$, then by the embedding $\dot{W}^{s,p}=\dot{F}^s_{p,p}\hookrightarrow \dot{F}^s_{p,2}= \dot{H}^{s,p}$, the inequality \eqref{GMO} is applicable to $f\in W^{s,p}(\mathbb{R}^N)$. On the other hand, if $p>2$, since the direction of the embedding is reversed, we need a stronger inequality. Fortunately, the same paper \cite{Besov} also gives a proof of an alternative inequality with the $\dot{B}^{s}_{p,p}$ norm, which we will adopt. \par
Proposition \ref{CC} can also be considered as a $p$-fractional generalization of Theorem 1.3 in \cite{doubly critical fractional}. However, let us note that a natural extension of the counterpart in \cite{doubly critical fractional} will give a statement for bounded sequences in $H^{s,p}$. Since we have $H^{s,p}\neq W^{s,p}$ in general when $p\neq 2,s\notin\mathbb{N}$, Proposition \ref{CC} is not merely a simple extension of Theorem 1.3 in \cite{doubly critical fractional}. \par
Now, let us give a proof for Proposition \ref{CC}.
\begin{proof}
As in \cite{Besov}, we have
\begin{align*}
\|f\|_{p_s^*}&\leq C\|f\|_{\dot{B}^{s}_{p,p}}^{\frac{p}{p_s^*}}\|f\|_{\dot{B}_{\infty,\infty}^{-N/p_s^*}}^{1-\frac{p}{p_s^*}} \\
&= C\|f\|_{\dot{F}^{s}_{p,p}}^{\frac{p}{p_s^*}}\|f\|_{\dot{B}_{\infty,\infty}^{-N/p_s^*}}^{1-\frac{p}{p_s^*}} \\
&= C'\|f\|_{\dot{W}^{s,p}}^{\frac{p}{p_s^*}}\|f\|_{\dot{B}_{\infty,\infty}^{-N/p_s^*}}^{1-\frac{p}{p_s^*}}.
\end{align*}
By Lemma 3.4 in \cite{Morrey}, the following embedding holds:
\[
\mathcal{M}_p^{p_s^*}\hookrightarrow \mathcal{M}_1^{p_s^*} \hookrightarrow \dot{B}_{\infty,\infty}^{-N/p_s^*}.
\]
Combining these, we obtain
\[
\|f\|_{p_s^*}\leq C\|f\|_{\dot{W}^{s,p}}^{\theta}\|f\|_{\mathcal{M}_p^{p_s^*}}^{1-\theta} \leq C\|f\|_{W^{s,p}}^{\theta}\|f\|_{\mathcal{M}_p^{p_s^*}}^{1-\theta}.
\] 
Apply this inequality for $f=u_n$. Then, by the definition of the Morrey norm, up to a subsequence, for any $n\in\mathbb{N}$, there exist $\rho_n>0$ and $y_n\in\mathbb{R}^N$ such that
\begin{equation}\label{morrey}
0<C<\rho_n^{-s}\|u_n\|_{L^p (B_{\rho_n}(y_n))}\leq C'\rho_n^{-s}.
\end{equation}
This implies that $\{\rho_n\}$ is bounded. Therefore, passing to a subsequence if necessary, we may assume $\rho_n\to \rho$ ($n\to\infty$). Define $v_n\coloneqq u_n(\cdot +y_n)$. Since $\{v_n\}$ is also bounded in $W^{s,p}$, up to a subsequence, we may assume 
\begin{align*}
v_n\rightharpoonup v \quad\text{weakly in $W^{s,p}(\mathbb{R}^N)$},
v_n\to v\quad\text{in $L^p_\mathrm{loc}(\mathbb{R}^N)$}.
\end{align*}
If $\rho=0$, then since $\{v_n\}$ is bounded in $L^{p_s^*}(\mathbb{R}^N)$, 
\[
\|v_n\|_{L^p(B_{\rho_n}(0))}\leq\|1\|_{L^{N/s}(B_{\rho_n}(0))} \|v_n\|_{L^{p_s^*}(B_{\rho_n}(0))} \leq C\|1\|_{L^{N/s}(B_{\rho_n}(0))}\to 0,
\]
which contradicts \eqref{morrey}. Therefore, $\rho>0$. Taking a limit in \eqref{morrey}, we obtain $v\neq 0$.
\end{proof}
Next, we prepare a lemma assuring the boundedness of $(PS)_c$ sequences. 
\begin{lemma}\label{bounded}
If $\{u_n\}$ is a $(PS)_c$ sequence for $I$, then $\{u_n\}$ is bounded.
\end{lemma}
\begin{proof}
\begin{equation}\label{I'[u_n]u_n}
\begin{split}
o(1)\cdot\|u_n\|&=I'[u_n]u_n \\
&= [u_n]_{s,p}^p+a\|u_n\|_p^p \\
&\phantom{=}-b\int_{\mathbb{R}^N}(K\ast F(u_n))f(u_n)u_n dx-\varepsilon_g \|u_n\|_{p_g}^{p_g},
\end{split}
\end{equation}
\begin{equation}
\begin{split}
c+o(1)&=I[u_n]\\
&= \frac{1}{p}([u_n]_{s,p}^p+a\|u_n\|_p^p)\\
&\phantom{=}-\frac{b}{2}\int_{\mathbb{R}^N}(K\ast F(u_n))F(u_n)dx-\frac{\varepsilon_g}{p_g}\|u_n\|_{p_g}^{p_g}.
\end{split}
\end{equation}
Therefore, we have
\begin{align*}
c+o(1)+o(1)\cdot \|u_n\|&\geq I[u_n]-\max\left\{\frac{1}{2\cdot p^\downarrow},\frac{1}{p_g}\right\} I'[u_n]u_n \\
&\geq \left(\frac{1}{p}-\max\left\{\frac{1}{2\cdot p^\downarrow},\frac{1}{p_g}\right\}\right) ([u_n]_{s,p}^p+a\|u_n\|_p^p) \\
&\geq C\|u_n\|^p.
\end{align*}
This implies $\{\|u_n\|\}$ is bounded.
\end{proof}

\section{Minimax values}
Let us consider $\tilde{a}\in L^\infty (\mathbb{R}^N;[0,\infty))$ with $0<\tilde{a}(x)\nearrow a$ ($|x|\to\infty$), where $\nearrow$ represents the convergence from below, and the functional
\begin{align*}
I_{\tilde{a}}[u] &=\frac{1}{p}[u]_{s,p}^p+\frac{1}{p}\int_{\mathbb{R}^N}\tilde{a}(x)|u|^p dx-\frac{b}{2}\int_{\mathbb{R}^N}(K\ast F(u))F(u)dx-\frac{\varepsilon_g}{p_g}\int_{\mathbb{R}^N}|u|^{p_g} dx.
\end{align*}
It is easy to check that $I_{\tilde{a}}$ has a mountain pass geometry. That is, the following holds:
\begin{lemma}\label{mountain pass geometry}
There exist positive constants $d$ and $\rho$ such that $I_{\tilde{a}}[u]\geq d$ if $\|u\|=\rho$, and there exists $v\in W^{s,p}(\mathbb{R}^N)$ with $\|v\|>\rho$ such that $I_{\tilde{a}}[v]\leq 0= I_{\tilde{a}} [0]$. 
\end{lemma}
\begin{proof}
By the Sobolev inequality, we have
\begin{align*}
I_{\tilde{a}}[u] &\geq C_1 \|u\|^p-C_2\|u\|^{2\cdot p^\uparrow} -C_3\|u\|^{2\cdot p^\downarrow}-C_4 \|u\|^{p^\downarrow+p^\uparrow}-C_5 \|u\|^{p_g}.
\end{align*}
Since $2\cdot p^\downarrow, p^\downarrow+p^\uparrow, 2\cdot p^\uparrow, p_g>p$, for $\rho>0$ sufficiently small, there exists $d>0$ such that $I_{\tilde{a}}[u]\geq d$ if $\|u\|=\rho$.\par
On the other hand, for fixed $u\in W^{s,p}(\mathbb{R}^N)$ and $t>0$,
\begin{align*}
I_{\tilde{a}}[tu] &\leq C_1' t^p\|u\|^p-C_2' t^{2\cdot p^\uparrow}\|u\|^{2\cdot p^\uparrow} -C_3' t^{2\cdot p^\downarrow}\|u\|^{2\cdot p^\downarrow} 
\\ &\phantom{=}\quad
-C_4' t^{p^\downarrow + p^\uparrow}\|u\|^{p^\downarrow+p^\uparrow}-C_5' t^{p_g} \|u\|^{p_g}.
\end{align*}
Hence, for $t>0$ sufficiently large, we have $I_{\tilde{a}}[tu] <0$. 
\end{proof}

Define
\[
c_{M;\tilde{a}}\coloneqq \inf_{\gamma\in \Gamma}\sup_{[0,1]}I_{\tilde{a}}\circ \gamma,
\]
where
\[
\Gamma\coloneqq \{\gamma\in C([0,1];W^{s,p}(\mathbb{R}^N))\mid \gamma(0)=0,I_{\tilde{a}} [\gamma(1)]<0\}.
\]
If $p<p_g<p_s^*$, let
\begin{align*}
c^*&=\min\left\{\left(\frac{1}{p}-\frac{1}{2\cdot p^\downarrow}\right) \left(\frac{p^\downarrow}{b}\right)^{\frac{p}{2\cdot p^\downarrow-p}} (a S^\downarrow)^{\frac{2\cdot p^\downarrow}{2\cdot p^\downarrow-p}}, 
\left(\frac{1}{p}-\frac{1}{2\cdot p^\uparrow}\right) \left(\frac{p^\uparrow}{b}\right)^{\frac{p}{2\cdot p^\uparrow-p}} (S^\uparrow)^{\frac{2\cdot p^\uparrow}{2\cdot p^\uparrow-p}}\right\}.
\end{align*}
If $p_g=p_s^*$, let
\[
c^*=\min\left\{\left(\frac{1}{p}-\frac{1}{2\cdot p^\downarrow}\right) \left(\frac{p^\downarrow}{b}\right)^{\frac{p}{2\cdot p^\downarrow-p}} (a S^\downarrow)^{\frac{2\cdot p^\downarrow}{2\cdot p^\downarrow-p}}, \left(\frac{1}{p}-\frac{1}{2\cdot p^\uparrow}\right) A(\varepsilon_g)\right\},
\]
where $A(\varepsilon_g)>0$ is the constant which satisfies
\[
(S^*)^{\frac{2\cdot p^\uparrow}{p}} A(\varepsilon_g)=\frac{b}{p^\uparrow} C_U A(\varepsilon_g)^{\frac{2\cdot p^\uparrow}{p}}+ \varepsilon_g (S^*)^{\frac{p_s^*}{p}} A(\varepsilon_g)^{\frac{p_s^*}{p}},
\]
\[
C_U\coloneqq\frac{1}{(S^*)^{\frac{N+\alpha}{ps}}} \int_{\mathbb{R}^N}(K\ast |U|^{p^\uparrow}) |U|^{p^\uparrow} dx
\]
for the extremal function $U$ corresponding to $S^*$ such that $[U]_{s,p}^p=\|U\|_{p_s^*}^{p_s^*}$. 

\begin{lemma}
Assume $N\geq \max\{2ps+\alpha,p^2 s\}$. If $p<p_g<p_s^*$, then we have $c_{M;\tilde{a}}<c^*$. Moreover, if $p_g=p_s^*$, then for $\varepsilon_g>0$ small enough, we have $c_{M;\tilde{a}}<c^*$. 
\end{lemma}
\begin{proof}
Let $u^\downarrow$ and $u^\uparrow$ be respectively the extremal functions corresponding to $S^\downarrow$ and $S^\uparrow$ such that
\[
\int_{\mathbb{R}^N}|u^\downarrow|^{p}dx= \int_{\mathbb{R}^N}(K\ast |u^\downarrow|^{p^\downarrow}) |u^\downarrow|^{p^\downarrow} dx
\]
and
\[
[u^\uparrow]_{s,p}^p= \int_{\mathbb{R}^N}(K\ast |u^\uparrow|^{p^\uparrow}) |u^\uparrow|^{p^\uparrow} dx.
\]
Then, by the definition of $S^\downarrow$, it automatically holds that $\|u^\downarrow\|_p^p=(S^\downarrow)^{\frac{2\cdot p^\downarrow}{2\cdot p^\downarrow-p}}$. Define
\[
u_\lambda^\downarrow (x)\coloneqq \lambda^{-N/p}u^\downarrow (x/\lambda),\quad u_\lambda^\uparrow (x)\coloneqq \lambda^{-\left(\frac{N}{p}-s\right)}u^\downarrow (x/\lambda).
\]
Let $t_\lambda^\downarrow,t_\lambda^\uparrow>0$ be such that $I_{\tilde{a}} [t_\lambda^\downarrow u_\lambda^\downarrow ]=\displaystyle\max_{t>0} I_{\tilde{a}} [t u_\lambda^\downarrow ]$, $I_{\tilde{a}} [t_\lambda^\uparrow u_\lambda^\uparrow ]=\displaystyle\max_{t>0} I_{\tilde{a}} [t u_\lambda^\uparrow ]$. By direct calculation, we have
\begin{align*}
0&=\left. \frac{d}{dt}\right|_{t=t_\lambda^\downarrow} I_{\tilde{a}} [t u_\lambda^\downarrow]\\
&= (t_\lambda^\downarrow)^{p-1} \lambda^{-ps}[u^\downarrow]_{s,p}^p+ (t_\lambda^\downarrow)^{p-1}\int_{\mathbb{R}^N}\tilde{a}(\lambda x)| u^\downarrow|^p dx \\
&\phantom{=} -\frac{b}{p^\downarrow} (t_\lambda^\downarrow)^{2\cdot p^\downarrow-1} \int_{\mathbb{R}^N}(K\ast |u^\downarrow|^{p^\downarrow}) |u^\downarrow|^{p^\downarrow} dx \\
&\phantom{=} -\frac{b}{p^\uparrow} (t_\lambda^\downarrow)^{2\cdot p^\uparrow-1} \lambda^{-\frac{2N\cdot p^\uparrow}{p}+N+\alpha}\int_{\mathbb{R}^N}(K\ast |u^\downarrow|^{p^\uparrow}) |u^\downarrow|^{p^\uparrow} dx \\
&\phantom{=}-\frac{b(p^\downarrow + p^\uparrow)}{p^\downarrow\cdot p^\uparrow} (t_\lambda^\downarrow)^{p^\downarrow + p^\uparrow-1} \lambda^{-\frac{N(p^\downarrow + p^\uparrow)}{p}+N+\alpha}\int_{\mathbb{R}^N}(K\ast |u^\downarrow|^{p^\downarrow}) |u^\downarrow|^{p^\uparrow} dx \\
&\phantom{=} -\varepsilon_g(t_\lambda^\downarrow)^{p_g-1} \lambda^{-\frac{N}{p}p_g+N}\int_{\mathbb{R}^N} |u^\downarrow|^{p_g} dx
\end{align*}
Note that $p-1<2\cdot p^\downarrow -1 \; (<p^\downarrow + p^\uparrow-1 <2\cdot p^\uparrow-1)$. If $t_\infty^\downarrow \coloneqq \displaystyle\limsup_{\lambda\to\infty} t_\lambda^\downarrow=\infty$, dividing the both side by $(t_\lambda^\downarrow)^{p-1}$ and taking the limit as $\lambda\to\infty$, we can see that the right hand side goes to $-\infty$ and get a contradiction. Therefore, $t_\infty^\downarrow <\infty$. Taking the limit as $\lambda\to\infty$ again, we obtain
\begin{align*}
0&=a (t_\infty^\downarrow)^{p-1} \int_{\mathbb{R}^N} |u^\downarrow|^{p} dx-\frac{b(t_\infty^\downarrow)^{2\cdot p^\downarrow-1}}{p^\downarrow} \int_{\mathbb{R}^N}(K\ast |u^\downarrow|^{p^\downarrow}) |u^\downarrow|^{p^\downarrow} dx \\
&=\left(a (t_\infty^\downarrow)^{p-1} -\frac{b(t_\infty^\downarrow)^{2\cdot p^\downarrow-1}}{p^\downarrow}\right) \int_{\mathbb{R}^N} |u^\downarrow|^{p} dx
\end{align*}
and thus $t_\infty^\downarrow=(a\cdot p^\downarrow /b)^{\frac{1}{2\cdot p^\downarrow-p}}$. \par
Note that if $N>2ps+\alpha$, then since $p^\uparrow<p$, we have
\begin{equation}\label{purpose of psharp<p}
-\frac{N(p^\downarrow + p^\uparrow)}{p}+N+\alpha=-\frac{N+\alpha}{N-ps}\cdot\frac{ps}{2}=-p^\uparrow s>-ps.
\end{equation}
Therefore, for $\lambda>0$ large enough, 
\begin{align*}
&\phantom{=}\left(a (t_\lambda^\downarrow)^{p-1} -\frac{b(t_\lambda^\downarrow)^{2\cdot p^\downarrow-1}}{p^\downarrow}\right) \int_{\mathbb{R}^N} |u^\downarrow|^{p} dx\\
&\geq (t_\lambda^\downarrow)^{p-1} \int_{\mathbb{R}^N} \tilde{a}(\lambda x)|u^\downarrow|^{p} dx-\frac{b(t_\lambda^\downarrow)^{2\cdot p^\downarrow-1}}{p^\downarrow} \int_{\mathbb{R}^N} |u^\downarrow|^{p} dx>0
\end{align*}
and thus
\[
a (t_\lambda^\downarrow)^{p-1} -\frac{b(t_\lambda^\downarrow)^{2\cdot p^\downarrow-1}}{p^\downarrow}>0.
\]
A direct calculation and \eqref{purpose of psharp<p} also yield
\begin{align*}
I_{\tilde{a}} [t_\lambda^\downarrow u_\lambda^\downarrow] &= \frac{(t_\lambda^\downarrow)^{p}}{p} \lambda^{-ps}[u^\downarrow]_{s,p}^p+ \frac{(t_\lambda^\downarrow)^{p}}{p}\int_{\mathbb{R}^N}\tilde{a}(\lambda x)| u^\downarrow|^p dx \\
&\phantom{=}-\frac{b}{2(p^\downarrow)^2} (t_\lambda^\downarrow)^{2\cdot p^\downarrow} \int_{\mathbb{R}^N}(K\ast |u^\downarrow|^{p^\downarrow}) |u^\downarrow|^{p^\downarrow} dx \\
&\phantom{=}-\frac{b}{2(p^\uparrow)^2} (t_\lambda^\downarrow)^{2\cdot p^\uparrow} \lambda^{-\frac{2N\cdot p^\uparrow}{p}+N+\alpha}\int_{\mathbb{R}^N}(K\ast |u^\downarrow|^{p^\uparrow}) |u^\downarrow|^{p^\uparrow} dx \\
&\phantom{=}-\frac{b}{p^\downarrow\cdot p^\uparrow} (t_\lambda^\downarrow)^{p^\downarrow + p^\uparrow} \lambda^{-\frac{N(p^\downarrow + p^\uparrow)}{p}+N+\alpha}\int_{\mathbb{R}^N}(K\ast |u^\downarrow|^{p^\downarrow}) |u^\downarrow|^{p^\uparrow} dx \\
&\phantom{=} -\frac{\varepsilon_g}{p_g}(t_\lambda^\downarrow)^{p_g} \lambda^{-\frac{N}{p}p_g+N}\int_{\mathbb{R}^N} |u^\downarrow|^{p_g} dx \\
&< \left(\frac{1}{p}a (t_\lambda^\downarrow)^{p} -\frac{b(t_\lambda^\downarrow)^{2\cdot p^\downarrow}}{2(p^\downarrow)^2}\right) \int_{\mathbb{R}^N} |u^\downarrow|^{p} dx \\
&\eqqcolon \ell_1(t_\lambda^\downarrow) \int_{\mathbb{R}^N} |u^\downarrow|^{p} dx.
\end{align*}
We already know $\ell_1'(t_\infty^\downarrow)=0$, $\ell_1'(t_\lambda^\downarrow)>0$ for $\lambda>0$ large enough, which implies
\begin{align*}
I_{\tilde{a}} [t_\lambda^\downarrow u_\lambda^\downarrow] &<\ell_1(t_\lambda^\downarrow) \int_{\mathbb{R}^N} |u^\downarrow|^{p} dx\leq \ell_1(t_\infty^\downarrow) \int_{\mathbb{R}^N} |u^\downarrow|^{p} dx \\
&= \left(\frac{1}{p}a (a\cdot p^\downarrow /b)^{\frac{p}{2\cdot p^\downarrow-p}} -\frac{b(a\cdot p^\downarrow /b)^{\frac{2\cdot p^\downarrow}{2\cdot p^\downarrow-p}}}{2(p^\downarrow)^2}\right)(S^\downarrow)^{\frac{2\cdot p^\downarrow}{2\cdot p^\downarrow-p}}\\
&=\left(\frac{1}{p}-\frac{1}{2\cdot p^\downarrow}\right) \left(\frac{p^\downarrow}{b}\right)^{\frac{p}{2\cdot p^\downarrow-p}} (a S^\downarrow)^{\frac{2\cdot p^\downarrow}{2\cdot p^\downarrow-p}}.
\end{align*}
Similarly,
\begin{align*}
0&=\left. \frac{d}{dt}\right|_{t=t_\lambda^\uparrow} I_{\tilde{a}} [t u_\lambda^\uparrow]\\
&= \left((t_\lambda^\uparrow)^{p-1} -\frac{b}{p^\uparrow} (t_\lambda^\uparrow)^{2\cdot p^\uparrow-1} \right)[u^\uparrow]_{s,p}^p+ (t_\lambda^\uparrow)^{p-1}\lambda^{ps}\int_{\mathbb{R}^N}\tilde{a}(\lambda x)| u^\uparrow|^p dx \\
&\phantom{=}-\frac{b}{p^\downarrow} (t_\lambda^\uparrow)^{2\cdot p^\downarrow-1} \lambda^{-\frac{2}{p}(N-ps)p^\downarrow+N+\alpha}\int_{\mathbb{R}^N}(K\ast |u^\uparrow|^{p^\downarrow}) |u^\uparrow|^{p^\downarrow} dx \\
&\phantom{=}-\frac{b(p^\downarrow + p^\uparrow)}{p^\downarrow\cdot p^\uparrow} (t_\lambda^\uparrow)^{p^\downarrow + p^\uparrow-1} \lambda^{-\frac{(N-ps)(p^\downarrow + p^\uparrow)}{p}+N+\alpha}\int_{\mathbb{R}^N}(K\ast |u^\uparrow|^{p^\downarrow}) |u^\uparrow|^{p^\uparrow} dx \\
&\phantom{=} -\varepsilon_g(t_\lambda^\uparrow)^{p_g-1} \lambda^{-\frac{N-ps}{p}p_g+N}\int_{\mathbb{R}^N} |u^\uparrow|^{p_g} dx.
\end{align*}
If $t_0^\uparrow \coloneqq \displaystyle\limsup_{\lambda\to+0} t_\lambda^\uparrow=\infty$, dividing the both side by $(t_\lambda^\uparrow)^{p-1}$ and taking the limit as $\lambda\to +0$, we can see that the right hand side goes to $-\infty$ and get a contradiction. Therefore, $t_0^\uparrow <\infty$. Taking the limit as $\lambda\to +0$ again, we obtain
\[
(t_0^\uparrow)^{p-1} -\frac{b}{p^\uparrow} (t_0^\uparrow)^{2\cdot p^\uparrow-1}=0,
\]
that is, $t_0^\uparrow=(p^\uparrow/b)^{\frac{1}{2\cdot p^\uparrow-p}}$. \par
In a similar way as above, noting that
\[
-\frac{(N-ps)(p^\downarrow + p^\uparrow)}{p}+N+\alpha=p^\downarrow s<ps,
\]
for $\lambda>0$ small enough, we can obtain
\[
0\leq \ell_2'(t_\lambda^\uparrow)\leq C\lambda^{p^\downarrow s}
\]
and hence, 
\[
\ell_2(t_\lambda^\uparrow)<\ell_2(t_0^\uparrow) + C'\lambda^{p^\downarrow s}\cdot (t_\lambda^\uparrow-t_0^\uparrow),
\]
where
\[
\ell_2(t)\coloneqq \frac{t^{p}}{p} -\frac{b}{2(p^\uparrow)^2} t^{2\cdot p^\uparrow}.
\]
Therefore, 
\begin{align*}
I_{\tilde{a}} [t_\lambda^\uparrow u_\lambda^\uparrow] &= \left(\frac{(t_\lambda^\uparrow)^{p}}{p} -\frac{b}{2(p^\uparrow)^2} (t_\lambda^\uparrow)^{2\cdot p^\uparrow} \right)[u^\uparrow]_{s,p}^p+ \frac{(t_\lambda^\uparrow)^{p}}{p}\lambda^{ps}\int_{\mathbb{R}^N}\tilde{a}(\lambda x)| u^\uparrow|^p dx \\
&\phantom{=}-\frac{b}{2(p^\downarrow)^2} (t_\lambda^\uparrow)^{2\cdot p^\downarrow} \lambda^{-\frac{2}{p}(N-ps)p^\downarrow+N+\alpha}\int_{\mathbb{R}^N}(K\ast |u^\uparrow|^{p^\downarrow}) |u^\uparrow|^{p^\downarrow} dx \\
&\phantom{=}-\frac{b}{p^\downarrow\cdot p^\uparrow} (t_\lambda^\uparrow)^{p^\downarrow + p^\uparrow} \lambda^{-\frac{(N-ps)(p^\downarrow + p^\uparrow)}{p}+N+\alpha}\int_{\mathbb{R}^N}(K\ast |u^\uparrow|^{p^\downarrow}) |u^\uparrow|^{p^\uparrow} dx \\
&\phantom{=} -\frac{\varepsilon_g (t_\lambda^\uparrow)^{p_g}}{p_g}\lambda^{-\frac{N-ps}{p}p_g+N}\int_{\mathbb{R}^N} |u^\uparrow|^{p_g} dx \\
&<\left(\frac{(t_\lambda^\uparrow)^{p}}{p} -\frac{b}{2(p^\uparrow)^2} (t_\lambda^\uparrow)^{2\cdot p^\uparrow} \right)[u^\uparrow]_{s,p}^p-C''\lambda^{p^\downarrow s} \\
&<(\ell_2(t_0^\uparrow) + C'\lambda^{p^\downarrow s}\cdot (t_\lambda^\uparrow-t_0^\uparrow)) [u^\uparrow]_{s,p}^p-C''\lambda^{p^\downarrow s} \\
&\nearrow \left(\frac{(t_0^\uparrow)^{p}}{p} -\frac{b}{2(p^\uparrow)^2} (t_0^\uparrow)^{2\cdot p^\uparrow} \right)[u^\uparrow]_{s,p}^p\quad (\lambda\to +0) \\
&= \left(\frac{1}{p}-\frac{1}{2\cdot p^\uparrow}\right) \left(\frac{p^\uparrow}{b}\right)^{\frac{p}{2\cdot p^\uparrow-p}} (S^\uparrow)^{\frac{2\cdot p^\uparrow}{2\cdot p^\uparrow-p}}.
\end{align*}

Let $U$ be the extremal function corresponding to $S^*$ such that $[U]_{s,p}^p=\|U\|_{p_s^*}^{p_s^*}$. Then, by the definition of $S^*$, automatically $\|U\|_{p_s^*}^{p_s^*} =(S^*)^{N/(ps)}$. Take $\varphi_\delta \in C_c^\infty(\mathbb{R}^N;[0,1])$ such that $\varphi_\delta=1$ in $B_\delta (0)$ and $\varphi_\delta=0$ in $\mathbb{R}^N\setminus B_{2\delta}(0)$. We define $U_\varepsilon (x)\coloneqq \varepsilon^{s-\frac{N}{p}}U(x/\varepsilon)$ and $u_\varepsilon^{*}\coloneqq \varphi_\delta U_\varepsilon$. Then, $[U_\varepsilon]_{s,p}^p=\|U_\varepsilon\|_{p_s^*}^{p_s^*}= (S^*)^{N/(ps)}$. We are now in a position to analyze the behavior of $u_\varepsilon^*$ as $\varepsilon\to +0$. \par
As in \cite{The Brezis-Nirenberg problem for the fractional p-Laplacian}, we know
\begin{align*}
[u_\varepsilon^*]_{s,p}^p &= (S^*)^{N/(ps)}+O(\varepsilon^{\frac{N-ps}{p-1}}), \\
\|u_\varepsilon^*\|_{p_s^*}^{p_s^*} &= (S^*)^{N/(ps)}+O(\varepsilon^{N/(p-1)}) \\
\|u_\varepsilon^*\|_p^p &\geq \left\{\begin{array}{ll} d\cdot\varepsilon^{ps} |\log\varepsilon |+O(\varepsilon^{ps}) &(N=sp^2)\\ d\cdot\varepsilon^{ps} +O(\varepsilon^{\frac{N-ps}{p-1}}) &(N>sp^2) \end{array}\right. .
\end{align*}
Moreover,
\begin{align*}
\left(\int_{\mathbb{R}^N}(K\ast |u_\varepsilon^*|^{p^\uparrow}) |u_\varepsilon^*|^{p^\uparrow} dx \right)^{\frac{p}{2\cdot p^\uparrow}} 
&\leq C(N,\alpha)^{\frac{p}{2\cdot p^\uparrow}} \|u_\varepsilon^*\|_{p_s^*}^p \\
&= C(N,\alpha)^{\frac{p}{2\cdot p^\uparrow}} ((S^*)^{N/(ps)}+O(\varepsilon^{N/(p-1)}))^{p/p_s^*} \\
&= C(N,\alpha)^{\frac{p}{2\cdot p^\uparrow}} (S^*)^{N/(p_s^* s)}+O(\varepsilon^{(N-ps)/(p-1)}) \\
&= C(N,\alpha)^{\frac{p}{2\cdot p^\uparrow}} (S^*)^{\frac{N}{ps}-1}+O(\varepsilon^{(N-ps)/(p-1)}).
\end{align*}
On the other hand, 
\begin{align*}
\int_{\mathbb{R}^N} (K\ast |u_\varepsilon^*|^{p^\uparrow}) |u_\varepsilon^*|^{p^\uparrow} dx 
&\geq \int_{B_\delta (0)} (K\ast |u_\varepsilon^*|^{p^\uparrow}|_{B_\delta (0)}) |u_\varepsilon^*|^{p^\uparrow} |_{B_\delta (0)}) dx \\
&= \int_{B_\delta (0)} \int_{B_\delta (0)} K(x-y)|U_\varepsilon(x)|^{p^\uparrow} |U_\varepsilon(y)|^{p^\uparrow}dxdy \\
&= \int_{\mathbb{R}^N} (K\ast |U_\varepsilon|^{p^\uparrow}) |U_\varepsilon|^{p^\uparrow} dx \\
&\phantom{=}-2 \int_{B_\delta (0)} \int_{\mathbb{R}^N\setminus B_\delta (0)} K(x-y)|U_\varepsilon(x)|^{p^\uparrow} |U_\varepsilon(y)|^{p^\uparrow}dxdy \\
&\phantom{=} -\int_{\mathbb{R}^N\setminus B_\delta (0)} \int_{\mathbb{R}^N\setminus B_\delta (0)} K(x-y)|U_\varepsilon(x)|^{p^\uparrow} |U_\varepsilon(y)|^{p^\uparrow}dxdy.
\end{align*}
Here, using Proposition \ref{asymptotic behavior}, we can see
\begin{align*}
&\phantom{=}\int_{B_\delta (0)} \int_{\mathbb{R}^N\setminus B_\delta (0)} K(x-y)|U_\varepsilon(x)|^{p^\uparrow} |U_\varepsilon(y)|^{p^\uparrow}dxdy \\
&\leq C\varepsilon^{2\cdot p^\uparrow(s-N/p)}\int_{B_\delta (0)} \int_{\mathbb{R}^N\setminus B_\delta (0)} K(x-y) (1+|x/\varepsilon|^{\frac{p}{p-1}})^{p^\uparrow (s-N/p)} (1+|y/\varepsilon|^{\frac{p}{p-1}})^{p^\uparrow (s-N/p)} dxdy \\
&\leq C\varepsilon^{2\cdot p^\uparrow(s-N/p)-2\frac{p}{p-1}p^\uparrow (s-N/p)}\int_{B_\delta (0)} \int_{\mathbb{R}^N\setminus B_\delta (0)} K(x-y) \\
&\phantom{=}\quad\quad\quad(\varepsilon^{\frac{p}{p-1}} +|x|^{\frac{p}{p-1}})^{p^\uparrow (s-N/p)} (\varepsilon^{\frac{p}{p-1}}+|y|^{\frac{p}{p-1}})^{p^\uparrow (s-N/p)} dxdy \\
&\leq O(\varepsilon^{\frac{N+\alpha}{p-1}})\left(\int_{\mathbb{R}^N\setminus B_\delta (0)}\frac{1}{(\varepsilon^{\frac{p}{p-1}} +|x|^{\frac{p}{p-1}})^N}dx\right)^{\frac{N+\alpha}{2N}} \left(\int_{B_\delta (0)}\frac{1}{(\varepsilon^{\frac{p}{p-1}} +|y|^{\frac{p}{p-1}})^N}dy\right)^{\frac{N+\alpha}{2N}} \\
&\leq O(\varepsilon^{\frac{N+\alpha}{p-1}})\left(\int_{\mathbb{R}^N\setminus B_\delta (0)}\frac{1}{|x|^{\frac{p}{p-1}N}}dx\right)^{\frac{N+\alpha}{2N}} \left(\int_{0}^{\delta}\frac{r^{N-1}}{(\varepsilon^{\frac{p}{p-1}} +r^{\frac{p}{p-1}})^N}dr\right)^{\frac{N+\alpha}{2N}} \\
&= O(\varepsilon^{\frac{N+\alpha}{2(p-1)}}) \left(\int_0^{\delta/\varepsilon}\frac{r^{N-1}}{(1+r^{\frac{p}{p-1}})^N}dr\right)^{\frac{N+\alpha}{2N}} \\
&\leq O(\varepsilon^{\frac{N+\alpha}{2(p-1)}}) \left(\int_0^{\infty}\frac{r^{N-1}}{(1+r^{\frac{p}{p-1}})^N}dr\right)^{\frac{N+\alpha}{2N}} \\
&\leq O(\varepsilon^{\frac{N+\alpha}{2(p-1)}}).
\end{align*}
Similarly,
\begin{align*}
&\phantom{=}\int_{\mathbb{R}^N\setminus B_\delta (0)} \int_{\mathbb{R}^N\setminus B_\delta (0)} K(x-y)|U_\varepsilon(x)|^{p^\uparrow} |U_\varepsilon(y)|^{p^\uparrow}dxdy \\
&\leq C\varepsilon^{2\cdot p^\uparrow(s-N/p)}\int_{\mathbb{R}^N\setminus B_\delta (0)} \int_{\mathbb{R}^N\setminus B_\delta (0)} K(x-y)(1+|x/\varepsilon|^{\frac{p}{p-1}})^{p^\uparrow (s-N/p)} 
(1+|y/\varepsilon|^{\frac{p}{p-1}})^{p^\uparrow (s-N/p)} dxdy \\
&\leq O(\varepsilon^{\frac{N+\alpha}{p-1}})\left(\int_{\mathbb{R}^N\setminus B_\delta (0)}\frac{1}{(\varepsilon^{\frac{p}{p-1}} +|x|^{\frac{p}{p-1}})^N}dx\right)^{\frac{N+\alpha}{2N}} 
\left(\int_{\mathbb{R}^N\setminus B_\delta (0)}\frac{1}{(\varepsilon^{\frac{p}{p-1}} +|y|^{\frac{p}{p-1}})^N}dy\right)^{\frac{N+\alpha}{2N}} \\
&\leq O(\varepsilon^{\frac{N+\alpha}{p-1}})\left(\int_{\mathbb{R}^N\setminus B_\delta (0)}\frac{1}{|x|^{\frac{p}{p-1}N}}dx\right)^{\frac{N+\alpha}{N}}\\
&= O(\varepsilon^{\frac{N+\alpha}{p-1}}).
\end{align*}
Eventually, 
\begin{align*}
&\phantom{=} \int_{\mathbb{R}^N}(K\ast |u_\varepsilon^*|^{p^\uparrow}) |u_\varepsilon^*|^{p^\uparrow} dx \\
&\geq C_U(S^*)^{\frac{N+\alpha}{ps}}-|O(\varepsilon^{\frac{N+\alpha}{2(p-1)}}) |-|O(\varepsilon^{\frac{N+\alpha}{p-1}})| \\
&= C_U(S^*)^{\frac{N+\alpha}{ps}}-|O(\varepsilon^{\frac{N+\alpha}{2(p-1)}}) |,
\end{align*}
where 
\begin{align*}
C_U &= \displaystyle\lim_{\varepsilon\to +0}\frac{1}{\|U_\varepsilon\|_{p_s^*}^{2\cdot p^\uparrow}} \int_{\mathbb{R}^N}(K\ast |U_\varepsilon|^{p^\uparrow}) |U_\varepsilon|^{p^\uparrow} dx \\ 
&= \frac{1}{\|U_\varepsilon\|_{p_s^*}^{2\cdot p^\uparrow}} \int_{\mathbb{R}^N}(K\ast |U_\varepsilon|^{p^\uparrow}) |U_\varepsilon|^{p^\uparrow} dx\quad (\forall \varepsilon>0) \\
&= \frac{1}{(S^*)^{\frac{N+\alpha}{ps}}} \int_{\mathbb{R}^N}(K\ast |U_\varepsilon|^{p^\uparrow}) |U_\varepsilon|^{p^\uparrow} dx\quad (\forall \varepsilon>0).
\end{align*}
(Note that $C_U$ is equal to $C(N,\alpha)$ when $p= 2$ because of the explicit formula while it is smaller than $C(N,\alpha)$ when $p\neq 2$.) \par
In the following, we consider the case $p_g=p_s^*$. \par
Let $t_\varepsilon^*>0$ be such that $\displaystyle\max_{t>0}I_{\tilde{a}}[t u_\varepsilon^*]=I_{\tilde{a}}[t_\varepsilon^* u_\varepsilon^*]$ and define $t_0^*\coloneqq\displaystyle\limsup_{\varepsilon\to +0}t_\varepsilon^*$. In the same way as above, we obtain
\[
(S^*)^{N/(ps)} (t_0^*)^{p-1}-\frac{b}{p^\uparrow}C_U(S^*)^{\frac{N+\alpha}{ps}}(t_0^*)^{2\cdot p^\uparrow-1}-\varepsilon_g (S^*)^{N/(ps)}(t_0^*)^{p_s^*-1}=0.
\]

Note that if $N\geq sp^2$, then $-\frac{(N-ps)(p^\downarrow + p^\uparrow)}{p}+N+\alpha =p^\downarrow s<ps\leq \frac{N-ps}{p-1}\leq \frac{N}{p-1}$ and $p^\downarrow s<\frac{N+\alpha}{2(p-1)}$ because $\frac{ps}{N}\leq\frac{1}{p}<\frac{1}{p-1}$. For $\varepsilon,\varepsilon'>0$ small enough, we can obtain
\begin{align*}
&\phantom{=}I_{\tilde{a}}[t_\varepsilon^* u_\varepsilon^*] \\
&= \frac{(t_\varepsilon^*)^{p}}{p}[u_\varepsilon^*]_{s,p}^p+ \frac{(t_\varepsilon^*)^{p}}{p}\int_{\mathbb{R}^N}\tilde{a}(\varepsilon x)| u_\varepsilon^* |^p dx \\
&\phantom{=}-\frac{b}{2(p^\downarrow)^2} (t_\varepsilon^*)^{2\cdot p^\downarrow} \varepsilon^{-\frac{2}{p}(N-ps)p^\downarrow+N+\alpha}\int_{\mathbb{R}^N}(K\ast |\varphi_\delta (\varepsilon \cdot) U|^{p^\downarrow}) | \varphi_\delta (\varepsilon x) U(x)|^{p^\downarrow} dx \\
&\phantom{=}-\frac{b}{p^\downarrow\cdot p^\uparrow} (t_\varepsilon^*)^{p^\downarrow + p^\uparrow} \varepsilon^{-\frac{(N-ps)(p^\downarrow + p^\uparrow)}{p}+N+\alpha}\int_{\mathbb{R}^N}(K\ast | \varphi_\delta (\varepsilon \cdot) U |^{p^\downarrow}) | \varphi_\delta (\varepsilon x) U(x) |^{p^\uparrow} dx \\
&\phantom{=}-\frac{b}{2(p^\uparrow)^2} (t_\varepsilon^*)^{2\cdot p^\uparrow} \int_{\mathbb{R}^N}(K\ast | u_\varepsilon^* |^{p^\uparrow}) | u_\varepsilon^* |^{p^\uparrow} dx-\frac{\varepsilon_g (t_\varepsilon^*)^{p_s^*}}{p_s^*}\| u_\varepsilon^*\|_{p_s^*}^{p_s^*} \\
&\leq\max_{t\in [t_\varepsilon^*-\varepsilon', t_\varepsilon^*+\varepsilon']}\left\{ \frac{t^{p}}{p}(S^*)^{N/(ps)}+O(\varepsilon^{\frac{N-ps}{p-1}})t^p+ O(\varepsilon^{ps})t^p \right. \\
&\phantom{=}-|O(\varepsilon^{-\frac{2}{p}(N-ps)p^\downarrow+N+\alpha})| t^{2\cdot p^\downarrow}-C\cdot \varepsilon^{p^\downarrow s} t^{p^\downarrow+p^\uparrow} \\
&\phantom{=}\left. -\frac{b}{2(p^\uparrow)^2} t^{2\cdot p^\uparrow} C_U(S^*)^{\frac{N+\alpha}{ps}}+ |O(\varepsilon^{\frac{N+\alpha}{2(p-1)}})| t^{2\cdot p^\uparrow} -\frac{\varepsilon_g}{p_s^*} t^{p_s^*}(S^*)^{\frac{N}{ps}}+O(\varepsilon^{\frac{N}{p-1}}) t^{p_s^*}\right\} \\
&\leq \max_{t>0}\left\{\frac{t^{p}}{p}(S^*)^{N/(ps)}-\frac{b}{2(p^\uparrow)^2} t^{2\cdot p^\uparrow} C_U(S^*)^{\frac{N+\alpha}{ps}}-\frac{\varepsilon_g}{p_s^*} t^{p_s^*}(S^*)^{\frac{N}{ps}}\right\} -C'\varepsilon^{p^\downarrow s} \\
&=\left. \left\{\frac{t^{p}}{p}(S^*)^{N/(ps)}-\frac{b}{2(p^\uparrow)^2} t^{2\cdot p^\uparrow} C_U(S^*)^{\frac{N+\alpha}{ps}}-\frac{\varepsilon_g}{p_s^*} t^{p_s^*}(S^*)^{\frac{N}{ps}}\right\} \right|_{t=t_0^*} -C'\varepsilon^{p^\downarrow s},
\end{align*}
where $C,C',\varepsilon'>0$ are constants independent of $\varepsilon_g>0$. By a direct computation, 
\begin{align*}
&\phantom{=}\lim_{\varepsilon_g\to +0}\left. \left\{\frac{t^{p}}{p}(S^*)^{N/(ps)}-\frac{b}{2(p^\uparrow)^2} t^{2\cdot p^\uparrow} C_U(S^*)^{\frac{N+\alpha}{ps}}-\frac{\varepsilon_g}{p_s^*} t^{p_s^*}(S^*)^{\frac{N}{ps}}\right\} \right|_{t=t_0^*(\varepsilon_g)} \\
&=\left(\frac{1}{p}-\frac{1}{2\cdot p^\uparrow}\right)\left(\frac{p^\uparrow}{b C_U}\right)^{\frac{p}{2\cdot p^\uparrow-p}}(S^*)^{\frac{2\cdot p^\uparrow}{2\cdot p^\uparrow-p}}
\end{align*}
On the other hand,
\[
\lim_{\varepsilon_g\to +0}A(\varepsilon_g)= \left(\frac{p^\uparrow}{b C_U}\right)^{\frac{p}{2\cdot p^\uparrow-p}}(S^*)^{\frac{2\cdot p^\uparrow}{2\cdot p^\uparrow-p}},
\]
and thus we reach the conclusion.
\end{proof}
In the following, we deal with the case where $\tilde{a}$ is a constant function. 

\begin{lemma}\label{float}
Let $a>0$ be a constant and $\{u_n\}$ be a $(PS)_c$ sequence of $I$ with $0<c<c^*$. Then, 
\[
\liminf_{n\to\infty}\|u_n\|_{p_s^*}>0,\quad \liminf_{n\to\infty}\|u_n\|_{p}>0.
\]
\end{lemma}
\begin{proof}
Suppose on the contrary that $\displaystyle \lim_{n\to\infty}\|u_n\|_{p_s^*}=0$. Then, by the Hardy-Littlewood-Sobolev inequality and the boundedness of $\{u_n\}$, we have
\[
\lim_{n\to\infty}\int_{\mathbb{R}^N}(K\ast |u_n|^{p^\uparrow}) |u_n|^{p^\uparrow}dx= \lim_{n\to\infty}\int_{\mathbb{R}^N}(K\ast |u_n|^{p^\downarrow}) |u_n|^{p^\uparrow}dx=0.
\]
Moreover, by the fractional Gagliardo-Nirenberg interpolation equality (or merely the interpolation using generalized H\"{o}lder's inequality), for $p<p_g\leq p_s^*$, there exists some $\theta\in [0,1]$ such that we have
\[
\|u_n\|_{p_g}\leq C\|u_n\|_{p_s^*}^\theta\|u_n\|_p^{1-\theta}=o(1).
\]
Therefore, from the definition of $(PS)_c$ sequence, we obtain
\begin{align}
c+o(1)&=I[u_n]=\frac{1}{p}([u_n]_{s,p}^p+a\|u_n\|_p^p)-\frac{b}{2(p^\downarrow)^2}\int_{\mathbb{R}^N}(K\ast |u_n|^{p^\downarrow}) |u_n|^{p^\downarrow}dx+o(1), \\
o(1)&=I'[u_n]u_n= [u_n]_{s,p}^p+a\|u_n\|_p^p -\frac{b}{p^\downarrow} \int_{\mathbb{R}^N}(K\ast |u_n|^{p^\downarrow}) |u_n|^{p^\downarrow}dx+o(1).\label{I' to 0}
\end{align}
These imply
\begin{equation}\label{c geq pp}
\begin{split}
c+o(1)&= \left(\frac{1}{p}-\frac{1}{2\cdot p^\downarrow}\right) ([u_n]_{s,p}^p+a\|u_n\|_p^p) \\
&\geq \left(\frac{1}{p}-\frac{1}{2\cdot p^\downarrow}\right) a\|u_n\|_p^p.
\end{split}
\end{equation}
By the definition of $S^\downarrow$ and \eqref{I' to 0}, we have
\begin{align*}
\|u_n\|_p^p &\geq S^\downarrow \left(\frac{p^\downarrow}{b}\right)^{\frac{p}{2\cdot p^\downarrow}}\left([u_n]_{s,p}^p+a\|u_n\|_p^p\right)^{\frac{p}{2\cdot p^\downarrow}}+o(1) \\
& \geq S^\downarrow \left(\frac{a\cdot p^\downarrow}{b}\right)^{\frac{p}{2\cdot p^\downarrow}}\left(\|u_n\|_p^p\right)^{\frac{p}{2\cdot p^\downarrow}} +o(1).
\end{align*}
This implies 
\[
\|u_n\|_p^p\geq \left(\frac{a\cdot p^\downarrow}{b}\right)^{\frac{p}{2\cdot p^\downarrow-p}}(S^\downarrow)^{\frac{2\cdot p^\downarrow}{2\cdot p^\downarrow-p}}+o(1).
\]
From this and \eqref{c geq pp}, we obtain $c\geq \displaystyle \left(\frac{1}{p}-\frac{1}{2\cdot p^\downarrow}\right) \left(\frac{p^\downarrow}{b}\right)^{\frac{p}{2\cdot p^\downarrow-p}}(a\cdot S^\downarrow)^{\frac{2\cdot p^\downarrow}{2\cdot p^\downarrow-p}}$, which contradicts $c<c^*$. Combining this with the fact that any subsequence of a $(PS)_c$ sequence also satisfies the $(PS)_c$ condition, we get $\displaystyle \liminf_{n\to\infty}\|u_n\|_{p_s^*}>0$. \par
Next, suppose on the contrary that $\displaystyle \lim_{n\to\infty}\|u_n\|_{p}=0$. In a similar way, we obtain
\begin{align}
c+o(1)&=I[u_n]=\frac{1}{p}[u_n]_{s,p}^p-\frac{b}{2(p^\uparrow)^2}\int_{\mathbb{R}^N}(K\ast |u_n|^{p^\uparrow}) |u_n|^{p^\uparrow}dx+\frac{\varepsilon_g}{p_g}\|u_n\|_{p_g}^{p_g}+o(1), \\
o(1)&=I'[u_n]u_n= [u_n]_{s,p}^p-\frac{b}{p^\uparrow} \int_{\mathbb{R}^N}(K\ast |u_n|^{p^\uparrow}) |u_n|^{p^\uparrow}dx-\varepsilon_g\|u_n\|_{p_g}^{p_g}+o(1).\label{I' to 0 with pg}
\end{align}
In the case $p_g\in (p,p_s^*)$, by the interpolation inequality, we have $\|u_n\|_{p_g}=o(1)$ and thus it follows that
\begin{equation}\label{sp to c}
\begin{split}
c+o(1)&= \left(\frac{1}{p}-\frac{1}{2\cdot p^\uparrow}\right) [u_n]_{s,p}^p
\end{split}
\end{equation}
and by the definition of $S^\uparrow$ and \eqref{I' to 0 with pg}, we have
\begin{align*}
[u_n]_{s,p}^p &\geq S^\uparrow\left(\int_{\mathbb{R}^N}(K\ast |u_n|^{p^\uparrow}) |u_n|^{p^\uparrow}dx \right)^{\frac{p}{2\cdot p^\uparrow}}+o(1) \\
&= S^\uparrow (p^\uparrow /b)^{\frac{p}{2\cdot p^\uparrow}}\left([u_n]_{s,p}^p\right)^{\frac{p}{2\cdot p^\uparrow}}+o(1),
\end{align*}
that is, 
\[
[u_n]_{s,p}^p\geq (p^\uparrow /b)^{\frac{p}{2\cdot p^\uparrow-p}}(S^\uparrow)^{\frac{2\cdot p^\uparrow}{2\cdot p^\uparrow-p}}+o(1).
\]
Combining this with \eqref{sp to c}, we reach a contradiction with $c<c^*$. \par
In the case where $p=2$ and $p_g=p_s^*$, since $1/2_s^*>1/(2\cdot 2^\uparrow)$, we have
\begin{equation}\label{c geq sp}
\begin{split}
c+o(1)&= \left(\frac{1}{p}-\frac{1}{2\cdot p^\uparrow}\right) [u_n]_{s,p}^p+\left(\frac{1}{p_s^*}-\frac{1}{2\cdot 2^\uparrow}\right) \int_{\mathbb{R}^N}(K\ast |u_n|^{p^\uparrow}) |u_n|^{p^\uparrow}dx \\
&\geq \left(\frac{1}{p}-\frac{1}{2\cdot p^\uparrow}\right) [u_n]_{s,p}^p.
\end{split}
\end{equation}
Moreover, by the definition of $S^\uparrow$ and \eqref{I' to 0 with pg}, using the fact that $S^*=C(N,\alpha)^{1/2^\uparrow}S^\uparrow$, we obtain
\begin{align*}
[u_n]_{s,p}^{2\cdot p^\uparrow} &\geq (S^\uparrow)^{\frac{2\cdot p^\uparrow}{p}} \int_{\mathbb{R}^N}(K\ast |u_n|^{p^\uparrow}) |u_n|^{p^\uparrow}dx \\
&= \frac{p^\uparrow}{b} (S^\uparrow)^{\frac{2\cdot p^\uparrow}{p}} ([u_n]_{s,p}^p-\varepsilon_g \|u_n\|_{p_s^*}^{p_s^*})+o(1) \\
&\geq \frac{p^\uparrow}{b} (S^\uparrow)^{\frac{2\cdot p^\uparrow}{p}} ([u_n]_{s,p}^p-\varepsilon_g (S^*)^{-p_s^* /p}[u_n]_{s,p}^{p_s^*})+o(1).
\end{align*}
This means
\[
(S^*)^{\frac{2\cdot p^\uparrow}{p}} [u_n]_{s,p}^{p}\leq \frac{b}{p^\uparrow} C(N,\alpha) ([u_n]_{s,p}^{p})^{\frac{2\cdot p^\uparrow}{p}}+ \varepsilon_g (S^*)^{\frac{p_s^*}{p}} ([u_n]_{s,p}^{p})^{\frac{p_s^*}{p}}+o(1),
\]
which implies $[u_n]_{s,p}^{p} \geq A(\varepsilon_g)+o(1)$. Combining this with \eqref{c geq sp}, we reach a contradiction with $c<c^*$.
\end{proof}

\begin{remark}
The same conclusion remains true with $B(\varepsilon_g)$ instead of $A(\varepsilon_g)$, where $B(\varepsilon_g)$ denotes the number such that
\[
1= \left(\frac{b}{p^\uparrow}\right)^{\frac{p}{2\cdot p^\uparrow}}(S^\uparrow)^{-1}(B(\varepsilon_g))^{1-\frac{p}{2\cdot p^\uparrow}} +\varepsilon_g^{\frac{p}{2\cdot p^\uparrow}} (S^*)^{-\frac{p_s^*}{2\cdot p^\uparrow}}(B(\varepsilon_g))^{\frac{p_s^*-p}{2\cdot p^\uparrow}}.
\]
In fact, we have $\displaystyle\lim_{\varepsilon_g\to +0}B(\varepsilon_g)= \lim_{\varepsilon_g\to +0}A(\varepsilon_g)$ and 
\begin{align*}
([u_n]_{s,p}^p)^{\frac{p}{2\cdot p^\uparrow}} &= \left(\frac{b}{p^\uparrow} \int_{\mathbb{R}^N}(K\ast |u_n|^{p^\uparrow}) |u_n|^{p^\uparrow}dx+\varepsilon_g\|u_n\|_{p_s^*}^{p_s^*}\right)^{\frac{p}{2\cdot p^\uparrow}}+o(1) \\
&\leq \left(\frac{b}{p^\uparrow} \int_{\mathbb{R}^N}(K\ast |u_n|^{p^\uparrow}) |u_n|^{p^\uparrow}dx \right)^{\frac{p}{2\cdot p^\uparrow}} +\left(\varepsilon_g\|u_n\|_{p_s^*}^{p_s^*}\right)^{\frac{p}{2\cdot p^\uparrow}}+o(1) \\
&\leq \left(\frac{b}{p^\uparrow}\right)^{\frac{p}{2\cdot p^\uparrow}}(S^\uparrow)^{-1}[u_n]_{s,p}^p +\left(\varepsilon_g (S^*)^{-p_s^* /p}([u_n]_{s,p}^p)^{p_s^* /p}\right)^{\frac{p}{2\cdot p^\uparrow}}+o(1).
\end{align*}
Here, we have used the fact that the function $\ell: [0,\infty)\to\mathbb{R}; t\mapsto t^{\frac{p}{2\cdot p^\uparrow}}$ is subadditive because it is concave and satisfies $\ell(0)=0$. 
This means
\[
1\leq \left(\frac{b}{p^\uparrow}\right)^{\frac{p}{2\cdot p^\uparrow}}(S^\uparrow)^{-1}([u_n]_{s,p}^p)^{1-\frac{p}{2\cdot p^\uparrow}} +\varepsilon_g^{\frac{p}{2\cdot p^\uparrow}} (S^*)^{-\frac{p_s^*}{2\cdot p^\uparrow}}([u_n]_{s,p}^p)^{\frac{p_s^*-p}{2\cdot p^\uparrow}}+o(1),
\]
which implies $[u_n]_{s,p}^p \geq B(\varepsilon_g)+o(1)$.
\end{remark}

\section{The proof of the main theorem}
Let $c_{N}$ denote the infimum of $I$ on the Nehari manifold $\mathcal{N}\coloneqq \{u\in W^{s,p}(\mathbb{R}^N)\setminus\{0\} \mid I'[u]u=0\}$. By the definition, if $c_{N}$ is a critical value, $c_{N}$ is the energy level of ground states. We can easily observe that $c_{N}$ is equal to the mountain pass energy level $c_{M;a}$ by comparing each of them with
\[
\inf_{u\in W^{s,p}(\mathbb{R}^N)\setminus\{0\}}\max_{t\geq 0} I[tu]
\]
in the same way as in the proof of Lemma 2.5 in \cite{doubly critical fractional}. To summarize, the following holds:
\begin{lemma}
\[
c_N=c_{M;a}=\inf_{u\in W^{s,p}(\mathbb{R}^N)\setminus\{0\}}\max_{t\geq 0} I[tu].
\]
\end{lemma}

Now, let us show our main theorem \ref{main}.
\begin{proof}
By the mountain pass lemma without the $(PS)$ condition, there exists a $(PS)_{c_{M;a}}$ sequence $\{u_n\}$. By Lemma \ref{bounded}, $\{u_n\}$ is bounded. By Lemma \ref{float} and Proposition \ref{CC}, there exists $\{y_n\}\subset\mathbb{R}^N$ such that $\{\tilde{u}_n\}=\{u_n(\cdot +y_n)\}$ converges weakly to some $\tilde{u}\neq 0$. From the translation-invariance of $I$ and the fact that for any $\varphi\in W^{s,p}(\mathbb{R}^N)$, we have
\begin{align*}
|I'[\tilde{u}_n]\varphi| &=|I'[u_n](\varphi(\cdot-y_n))| \\
&\leq \|I'[u_n]\|_{W^{-s,p'}}\|\varphi(\cdot-y_n)\|_{W^{s,p}}=o(1)\cdot \|\varphi\|_{W^{s,p}},
\end{align*}
we deduce that $\{\tilde{u}_n\}$ is also a $(PS)_{c_{M;a}}$ sequence. Since each term of the formula of $I'$ is weak-to-weak continuous, $\tilde{u}$ is a (nontrivial) weak solution of \eqref{cho}. \par
Moreover, by the Brezis-Lieb splitting property, 
\begin{align*}
c_{N} &\leq I[\tilde{u}]=I[\tilde{u}]-\max\left\{\frac{1}{2\cdot p^\downarrow}, \frac{1}{p_g}\right\} I'[\tilde{u}] \tilde{u} \\
&=\lim_{n\to\infty}(I[\tilde{u}_n]-\max\left\{\frac{1}{2\cdot p^\downarrow}, \frac{1}{p_g}\right\} I'[\tilde{u}_n] \tilde{u}_n)\\
&\phantom{=}-\lim_{n\to\infty}(I[\tilde{u}_n-\tilde{u}]-\max\left\{\frac{1}{2\cdot p^\downarrow}, \frac{1}{p_g}\right\} I'[\tilde{u}_n-\tilde{u}] (\tilde{u}_n-\tilde{u})) \\
&\leq \lim_{n\to\infty}(I[\tilde{u}_n]-\max\left\{\frac{1}{2\cdot p^\downarrow}, \frac{1}{p_g}\right\} I'[\tilde{u}_n] \tilde{u}_n) \\
&= \lim_{n\to\infty}I[\tilde{u}_n]+0 \\
&=c_{M;a}.
\end{align*}
Combining this with $c_{N}=c_{M;a}$, we deduce $I[\tilde{u}] = c_{N}$, which implies that $\tilde{u}$ is a ground state solution of \eqref{cho}. Consequently, we can also obtain
\begin{align*}
0&=\lim_{n\to\infty}(I[\tilde{u}_n-\tilde{u}]-\max\left\{\frac{1}{2\cdot p^\downarrow}, \frac{1}{p_g}\right\} I'[\tilde{u}_n-\tilde{u}] (\tilde{u}_n-\tilde{u})) \\
&=\lim_{n\to\infty}\left(\frac{1}{p}-\max\left\{\frac{1}{2\cdot p^\downarrow}, \frac{1}{p_g}\right\}\right)([\tilde{u}_n-\tilde{u}]_{s,p}^p+a \|\tilde{u}_n-\tilde{u}\|_{p}^{p}) \\
&\phantom{=}+ \lim_{n\to\infty}b\left(\frac{1}{p^\uparrow} \max\left\{\frac{1}{2\cdot p^\downarrow}, \frac{1}{p_g}\right\}-\frac{1}{2(p^\uparrow)^2}\right)\int_{\mathbb{R}^N}(K\ast | \tilde{u}_n-\tilde{u} |^{p^\uparrow}) | \tilde{u}_n-\tilde{u} |^{p^\uparrow}dx\\
&\phantom{=}+ \lim_{n\to\infty}b\left(\left(\frac{1}{p^\uparrow}+\frac{1}{p^\downarrow}\right) \max\left\{\frac{1}{2\cdot p^\downarrow}, \frac{1}{p_g}\right\}-\frac{1}{p^\downarrow p^\uparrow}\right) 
\int_{\mathbb{R}^N}(K\ast | \tilde{u}_n-\tilde{u} |^{p^\downarrow}) | \tilde{u}_n-\tilde{u} |^{p^\uparrow}dx \\ 
&\phantom{=}+ \lim_{n\to\infty}b\left(\max\left\{\frac{1}{2\cdot p^\downarrow}, \frac{1}{p_g}\right\}-\frac{1}{2\cdot p^\downarrow}\right) \int_{\mathbb{R}^N}(K\ast | \tilde{u}_n-\tilde{u} |^{p^\downarrow}) | \tilde{u}_n-\tilde{u} |^{p^\downarrow}dx \\
&\phantom{=}+ \lim_{n\to\infty}\varepsilon_g\left(\max\left\{\frac{1}{2\cdot p^\downarrow}, \frac{1}{p_g}\right\}-\frac{1}{p_g}\right)\|\tilde{u}_n-\tilde{u}\|_{p_g}^{p_g} \\
&\geq \lim_{n\to\infty}\left(\frac{1}{p}-\max\left\{\frac{1}{2\cdot p^\downarrow}, \frac{1}{p_g}\right\}\right)([\tilde{u}_n-\tilde{u}]_{s,p}^p+a \|\tilde{u}_n-\tilde{u}\|_{p}^{p})
\end{align*}
and thus $\{\tilde{u}_n\}$ converges strongly to $\tilde{u}$.
\end{proof}
\noindent\textbf{Funding}\quad No funding was received to assist with the preparation of this manuscript.\vspace{0pt}\\

\noindent\textbf{Disclosure statement}\quad The author report there are no competing interests to declare. \vspace{0pt}\\

\noindent\textbf{Data Availability}\quad Data sharing is not applicable to this article as no datasets were generated or analysed in the current study.

\end{document}